\documentclass[11pt]{article}

\usepackage[a4paper,top=3cm,bottom=3.5cm,left=3cm,right=3cm]{geometry}

\usepackage{amssymb}
\usepackage{amsmath}
\usepackage{amsthm}
\usepackage[english]{babel}
\usepackage[latin1]{inputenc}
\usepackage{latexsym}
\usepackage{mathrsfs}
\usepackage{bbm}
\usepackage{graphicx}
\usepackage[hidelinks]{hyperref}
\usepackage{accents}
\usepackage{cases}
\usepackage{url}
\usepackage{color}

\numberwithin{equation}{section}

\newtheorem{thm}{Theorem}[section]
\newtheorem{prop}[thm]{Proposition}
\newtheorem{defn}[thm]{Definition}
\newtheorem{lemma}[thm]{Lemma}

\theoremstyle{Definition}
\newtheorem{remark}[thm]{Remark}

\newcommand{\nn}{\mathbb{N}}
\newcommand{\rr}{\mathbb{R}}

\newcommand{\eee}{\mathbb{E}_{x}}

\newcommand{\aaa}{\mathcal{A}}
\newcommand{\mm}{\mathcal{M}}

\newcommand{\vphi}{\varphi}

\newcommand{\ubar}[1]{\underaccent{\bar}{#1}}
\newcommand{\x}{\ubar x }
\newcommand{\xx}{\bar x }
\newcommand{\yy}{\bar y }

\newcommand{\q}{\theta}

\newcommand{\uu}{\mathcal{U}}

\DeclareMathOperator*{\argmax}{arg\,max}

\title{\textbf{Optimal price management in retail energy markets: an impulse control problem with asymptotic estimates}}

\author{Matteo Basei\footnote{University of California, Berkeley, IEOR department, \sf basei at berkeley dot edu.}}

\begin{document}

\maketitle

\begin{abstract} 
\noindent We consider a retailer who buys energy in the wholesale market and resells it to final consumers. The retailer has to decide when to intervene to change the price he asks to his customers, in order to maximize his income. We model the problem as an infinite-horizon stochastic impulse control problem. We characterize an optimal price strategy and provide analytical existence results for the equations involved. We then investigate the dependence on the intervention cost. In particular, we prove that the measure of the continuation region is asymptotic to the fourth root of the cost. Finally, we provide some numerical results and consider a suitable extension of the model.

\vspace{5mm}

\vspace{3mm}
\end{abstract}

\section{Introduction}
\label{sec:intro}

In energy markets, retailers first buy energy in the wholesale market and then resell it to final consumers. When deciding his price strategy, a retailer has to consider several elements. Indeed, high prices correspond to high unitary incomes from a small number of customers, whereas low prices lead to low unitary incomes from a greater market share. Moreover, the operational costs can overcome the income from the sale of energy. Finally, adjusting the price implies a cost, so that the intervention times have to be carefully evaluated. In this paper, we propose an original model, based on discrete-time interventions on a continuous-time stochastic process, and determine when and how a representative retailer should intervene in order to maximize his earnings. 

\vspace{0.5cm}

As energy is traded almost instantaneously, the wholesale price of energy can be modelled as a continuous-time stochastic process. Conversely, it is reasonable to represent the final price of energy by a piecewise-constant process, since all the customers have to be informed before each price adjustment, due to specific clauses in the contracts. Hence, the retailer can adjust the price only by means of discrete-time interventions (impulse controls). More in detail, let $X_t$ be the spread between the final price and the wholesale price of the commodity, i.e., the net income when selling one unit of energy. The retailer's market share is modelled by $\Phi(X_t)$, for a suitable function $\Phi \in [0,1]$. The retailer's payoff consists in the income from the sale of energy and in the operational costs, here assumed to be quadratic with respect to the market share $\Phi(X_t)$. Finally, we assume a constant intervention penalty, namely, $c >0$. Hence, if $u=\{(\tau_k,\delta_k)\}_{k \geq 1}$ denotes the retailer's intervention policy, i.e., the intervention times and the corresponding shifts in the price process, we deal with the following impulse control problem:
\begin{equation}
\label{intro}
V(x)=\sup_{u \in \uu_x} \eee \bigg[ \int_0^\infty e^{-\rho t} \left( X_t \Phi(X_t) - b \Phi(X_t)^2 \right) dt - \sum_{k \in \nn} e^{- \rho \tau_k} c \bigg],
\end{equation}
where $\uu_x$ is the class of admissible controls and $X_t$ is a scaled Brownian motion with jumps $\delta_k$ at times $\tau_k$. We refer to Section \ref{sec:formulation} for precise definitions and remarks on the model.

To study the problem in \eqref{intro}, we use a verification approach: if a function is regular enough and satisfies a suitable quasi-variational inequality, it coincides with the value function of the problem and can be used to characterize an optimal control. The procedure we follow consists in three steps. First, we make an educated guess $\tilde V$ for the value function $V$, by heuristically solving the quasi-variational inequality. The candidate $\tilde V$ depends on five parameters, solution to a suitable system of  algebraic equations (smooth-fit conditions). In the second step of the procedure, we analytically prove that the parameter system admits a unique solution, so that $\tilde V$ is well-defined. The final step consists in proving that all the assumptions of the verification theorem are satisfied. We then get a characterization of an optimal control: the retailer should intervene when the state variable exits from the region $]\x,\xx[$ (continuation region), moving the process to a suitable optimal state $x^*$, where $x^*$ is explicitly known and the thresholds $\x,\xx$ are defined by a system of algebraic equations.

In the second part of the paper, we investigate the properties of the optimal control and of the value function with respect to the intervention cost $c>0$. When $c$ decreases, intervening is less expensive, so that the continuation region $]\x(c), \xx(c)[$ (we here specify the dependence on $c$) gets smaller, finally degenerating into the singleton $\{x^*\}$ as $c \to 0^+$. Moreover, the problem is not robust in a right neighbourhood of zero, in the sense that $dV/dc$ diverges as $c \to 0^+$, with $V$ as in \eqref{intro}. As for the continuation region $]\x(c), \xx(c)[$, we can actually go farther, proving an asymptotic estimate for this convergence. Namely, the continuation region converges to its limit as the fourth root of the intervention cost:
\begin{equation}
\label{intro2}
\x(c) \sim_{c \to 0^+} x^* - C \sqrt[4]{c}, \qquad\quad \xx(c) \sim_{c \to 0^+} x^* + C \sqrt[4]{c},
\end{equation}
where $C>0$. In particular, we provide an explicit expression for $C$, in terms of the parameters of the problems. To the best of our knowledge, this is the first example  of explicit asymptotic estimates for the continuation region of an impulse control problem. 

Finally, we generalize our model by adding a drift term in the underlying process and a state-dependent term in the intervention cost. We correspondingly adapt the verification procedure and provide sufficient conditions to characterize an optimal control. In this extended framework, analytical results are not possible for the system of equations, and we propose numerical solutions for two set of parameters.

\vspace{0.5cm}

In the past, the price set by a company for a product or service was relatively static in time. In the last decades, the advent of new technologies allowed quicker price adjustments and a better understanding of the customers' demand, leading to a dynamic pricing system. Correspondingly, a vast literature has arisen on the strategies to select the price strategy which maximizes the revenues (yield management problems). We cite the seminal papers Belobaba \cite{Belobaba87}, McGill and van Ryzin \cite{McGillRyzin99}, and refer the reader to the comprehensive survey in Elmaghraby and Keskinocak \cite{ElmKesk03}. Pricing problems for contracts in the energy industry are considered in Aïd \cite[Section 4.4]{Aid15}, with a detailed literature review on the subject. Quite surprisingly, existing papers mainly focus on big industrial customers, and the case of a retailer selling energy to smaller final consumers has not been intensively studied: to our knowledge, this is the first contribution to the subject.

A key-point in our approach is the use of impulse controls. In impulse control theory, controllers intervene by discrete-time shifts, moving the process from the current state to a more convenient one. Conversely, in standard control theory, the underlying process is affected by continuous-time interventions. In many applications, impulse controls provide more realistic models with respect to classic controls, since it is often the case that controllers can intervene on the process only a discrete number of times. Considering, for example, the case of optimal harvesting problems, optimal dividend policies, or the pricing problem described in this article, continuous-time controls would not fit the actual nature of the problems, whereas impulse controls are much closer to the real situation. Besides the vast range of known applications, impulse control theory is also theoretically challenging, as the Hamilton-Jacobi-Bellman equation characterizing standard control theory is here substituted by a more demanding quasi-variational inequality. For these reasons, impulse control theory has become a rich research subject in the latest years, both from a theoretical point of view and in terms of practical applications.

For a general introduction to impulse problems, we refer the reader to Øksendal and Sulem \cite[Chap.~6]{OksendalSulem07}. For a viscosity approach to the quasi-variational inequality characterizing the value function of impulse problems, see Davis et al.~\cite{DavisGuoWu10}, Guo and Wu \cite{GuoWu09}. Øksendal \cite{Oksendal99}, Øksendal et al.~\cite{OksendalUboeZhang07} study the robustness of a class of impulse problems with respect to the intervention costs. The extension of impulse control theory to the case where several players are present, i.e., impulse game theory, has been studied by Cosso \cite{Cosso13} in the zero-sum case and recently by Aïd et al.~\cite{AidBaseiCampiCallegaroVargiolu}, Basei et al.~\cite{BaseiCaoGuo}, Ferrari and Koch \cite{FerrariKoch18} in the nonzero-sum case. Federico et al.~\cite{FedeRoseTacc2018} propose an alternative approach to deal with stationary one-dimensional impulse control problems. Impulse controls have been used to model optimization problems for exchange and interest rates, see Cadenillas and Zapatero \cite{CadenZapatero99}, Jeanblanc-Picqué \cite{Jeanblanc93}, Mitchell et al.~\cite{MitchellFengMuth14}, Mundaca and Øksendal \cite{MundacaOksendal98}. For dividend policy optimization, we refer to Alvarez and Lempa \cite{AlvarezLempa08}, Cadenillas et al.~\cite{CadenChoulliTaksarZhang06}, Cadenillas et al.~\cite{CadenSarkZap07}, Jeanblanc-Picqué and Shiryaev \cite{JeanblancShir95}. As for inventory problems, we cite Bensoussan et al.~\cite{BensoussanMoussawuCaka10}, Cadenillas et al.~\cite{CadenLaknerPinero10}, Harrison et al.~\cite{HarrisonSelkeTaylor83}. Impulse controls also model problems in optimal harvesting: we refer the reader to Alvarez \cite{Alvarez04}, Willassen \cite{Willassen98}. Finally, for portfolio optimization see, e.g., Korn \cite{Korn98}, Korn \cite{Korn99}.

\vspace{0.5cm}

The contribution of this paper is twofold: financial and mathematical. On the one hand, from a financial point of view, we propose a tractable model for the optimization problem faced by a retailer setting the price policy to sell his customers the energy he bought in the wholesale market. In so doing, we aim at filling a gap in the literature, as detailed above. Key-points in our approach are the use of impulse controls and the presence of analytical existence results with semi-explicit formulas. In particular, we would like to underline the potentialities of impulse controls applied to problems in energy markets: although technically more demanding than standard controls, they better incorporate some features of the markets. On the other hand, from a mathematical point of view, our main contribution consists in \eqref{intro2}. To our knowledge, explicit asymptotic estimates for the continuation region of an impulse control problem have never been investigated and proved before. In particular, in Remark \ref{rem:extest} we extend \eqref{intro2} to any impulse problem with constant intervention costs, Brownian underlying, quadratic symmetric payoff, i.e., to any problem in the form
\begin{equation}
\label{intro3}
\sup_{u \in \uu_x} \eee \bigg[ \int_0^\infty e^{-\rho t} (X_t - x_v)^2 dt - \sum_{k \in \nn} e^{- \rho \tau_k} c \bigg],
\end{equation}
with $X_t$ as above and for suitable constants $\sigma, c >0$, $x_v \in \rr$. This provides a better insight on the properties of impulse problems and leaves room for further research, as problems in the form \eqref{intro3} are common in impulse control theory, see Øksendal and Sulem \cite[Chap.~7]{OksendalSulem07} and the references therein. 

\vspace{0.5cm}

The structure of the paper is as follows. In Section \ref{sec:formulation} we provide a precise formulation of the problem and its mathematical framework. Section \ref{sec:verifthm} provides a verification theorem, i.e., sufficient conditions to characterize the value function and an optimal control. By means of such conditions, in Section \ref{sec:optimalcontrol} we make an educated guess for the value function, prove that it is well-defined (the definition involves a system of algebraic equations), apply the verification theorem to characterize an optimal control. Section \ref{sec:estimates} collects several estimates with respect to the intervention cost $c>0$ and, in particular, an asymptotic estimate for the continuation region as $c \to 0^+$. Extensions of the model are proposed in Section \ref{sec:extensions}. Finally, Section \ref{sec:conclusions} concludes.

\section{Formulation of the problem}
\label{sec:formulation}

We consider a retailer who buys energy (e.g., electricity, gas, gasoline) in the wholesale market and resells it to final consumers. We model the wholesale price of the commodity (i.e., what the retailer pays for one unit of energy) by a scaled Brownian motion:
\begin{equation}
\label{defS}
S_t = s + \sigma W_t,
\end{equation}
for $t\geq 0$, where $s > 0$ is the initial wholesale price, $\sigma>0$ is a fixed constant and $W$ is a real Brownian motion defined on a filtered probability space $(\Omega, \mathcal{F}, \{\mathcal{F}_t\}_{t \geq 0}, \mathbb{P})$, where $\{\mathcal{F}_t\}_{t \geq 0}$ is the natural filtration of $W$. For the Brownian motion as a model for the electricity price, see Benth and Koekebakker \cite{BenthKoek08} and Remark \ref{rem:comm} below. Also, we refer to Section \ref{sec:extensions} for a generalization of \eqref{defS} including a drift term. Notice that the retailer has no control on the wholesale price: in most of the cases, i.e., when the company is not too big, this is a reasonable assumption. 

The retailer resells the energy he bought to final consumers. According to the most common contracts in energy markets, a retailer can change the price only after a written communication to his customers. Hence, we model the final price of the commodity (i.e., what consumers actually find in their bill for one unit of energy) by a piecewise-constant process $P$. Namely, we consider an initial price $p>0$, a sequence $\{\tau_k\}_{k \geq 1}$ of increasing random times, corresponding to the retailer's interventions to adjust the price, and a sequence $\{\delta_k\}_{k \geq 1}$ of real-valued impulses, corresponding to the jumps of the process (the retailer can both increase and decrease the price). For every $t \geq 0$, we then have 
\begin{equation}
\label{defP}
P_t = p + \sum_{\tau_k \leq t} \delta_k.
\end{equation}
In Definition \ref{def:admissiblecontrols} below, we provide precise conditions on the variables $\tau_k$ and $\delta_k$. We also assume that the retailer faces a fixed cost $c > 0$ when intervening to adjusts the price $P_t$. We refer to Section \ref{sec:extensions} for a generalization including a state-dependent term in the intervention costs.

Let us denote by $X_t$ the difference between the final price $P_t$ and the wholesale price $S_t$, at time $t \geq 0$. In other words, $X_t$ represents the retailer's unitary net income from the sale of energy. By \eqref{defS} and \eqref{defP}, we have
\begin{equation}
\label{defX}
X_t = P_t - S_t = x - \sigma W_t + \sum_{\tau_k \leq t} \delta_k,
\end{equation}
for each $t \geq 0$, where we have set $x = p-s$.

It is reasonable to assume that the retailer's market share at time $t\geq 0$ is a function of the spread $X_t$, denoted by $\Phi(X_t)$. Here, $\Phi$ is a suitable function with values in $[0,1]$, where $\Phi(X_t)=0$ means that the retailer has no customers and $\Phi(X_t) = 1$ corresponds to a monopolistic position. We underline the importance of linking the market share to the spread $X_t$, rather than to the final price $P_t$. Indeed, assuming that the market share only depends on $P_t$ would be misleading: if the wholesale price $S_t$ is high, all the retailers in the market would ask a high final price, so that high values of $P_t$ very close to $S_t$ would still be advantageous from the consumers' point of view.

We now model the function $\Phi$. First, we notice that, if the spread $X_t$ increases, the retailer has a bigger margin on the sale of energy, which may result in the loss of some customers; hence, the function $\Phi$ has to be decreasing. Moreover, if $X_t$ is close to zero, that is, if the retailer sells energy at the lowest price possible for him, the market share will be close to $1$; hence, we expect $\Phi(0)=1$. Finally, it is reasonable to assume that, if the spread is too big, say $X_t \geq \Delta$, the retailer loses all his customers; hence, we assume $\Phi(\Delta)=0$, where $\Delta >0$ is a fixed constant. A simple function with this properties (also see Remark \ref{rem:comm}) is given by
\begin{equation}
\label{defPHI}
\Phi(x) = 
\begin{cases}
1, & x \leq 0, \\
- \frac{1}{\Delta}(x-\Delta), & 0 < x < \Delta,\\
0, & x \geq \Delta,
\end{cases}
\end{equation}
for every $x \in \rr$. In other words, the market share is here a truncated linear function of $X_t$, with two thresholds: if $X_t \leq 0$ all the customers buy energy from the retailer, whereas if $X_t \geq \Delta$ the retailer has lost all his customers. 

Managing customers implies some continuous-time operating costs faced by the retailer (customer service, technical structures,...). Such costs are clearly increasing with respect to the market share $\Phi(X_t)$: the more customers the retailer has, the bigger the operating costs are. Indeed, as the number of customers increases, more and more investments are needed (a bigger structure to manage customer service,...), not proportional to the increase in the market share. For this reason, we consider increasing returns to scale, and assume that the operating costs are quadratic with respect to the market share, i.e., given by $b\Phi(X_t)^2$, where $b \geq 0$ is a fixed constant. In particular, we underline that the case $b=0$ is possible and consistent with the model.

The retailer's cash flow at time $t \geq 0$ consists in the income from the sale of energy $X_t\Phi(X_t)$ (unitary income multiplied by the market share) and in the operating costs $b\Phi(X_t)^2$. Hence, the instantaneous payoff in $t \geq 0$ is given by $R(X_t)$, with
\begin{equation}
\label{defR}
R(x) = x\Phi(x) - b \Phi(x)^2=
\begin{cases}
x - b, & \text{if $x<0$,}\\
f(x), & \text{if $0 \leq x \leq \Delta$,}\\
0, & \text{if $x > \Delta$,}
\end{cases}
\end{equation}
where $x \in \rr$ is the current state of the process and $f$ is a concave parabola, 
\begin{equation}
\label{defFparab}
f(x) = - \alpha (x-x_v)^2+ y_v,
\qquad\,\,
\alpha = \frac{\Delta + b}{\Delta^2}, 
\quad\!
x_v = \frac{\Delta(\Delta + 2b)}{2 (\Delta +b)}, 
\quad\!
y_v = \frac{\Delta^2}{4( \Delta +b )}.
\end{equation}
From an economical point of view, we remark the following properties of the payoff $R$.
\begin{itemize}
	\item[-] The payoff is bounded from above, with $\max_{x \in \rr} R(x) = R(x_v) = y_v$. Hence, the optimal static state is $x_v$. Notice that the corresponding market share is 
	\begin{equation}
	\label{defPhiB}
	\Phi_v=\Phi(x_v) = \frac{\Delta}{2(\Delta +b)}.
	\end{equation}
	In particular, if $b=0$ the optimal share is $1/2$.
	
	\item[-] The payoff is positive, $R(X_t) \geq 0$, if and only if $X_t \in [x_z, \Delta]$, where 
	\begin{equation}
	\label{defzeroparab}
	x_z = \frac{2b\Delta}{2(\Delta+b)}.
	\end{equation}
	In other words, if we want the income from the sale of energy to be higher than the operational costs, we need the spread between the wholesale price and the final price to be greater than $x_z$.
	
	\item[-] If we consider $x_v,y_v,x_z, \Phi_v$ as functions of $b$, we notice that 
	\begin{equation}
	\label{limitbehav}
	x_v'(b)>0, \qquad\quad
	y_v'(b)<0, \qquad\quad
	x_z'(b)>0, \qquad\quad
	\Phi_v'(b)<0,
	\end{equation}
	for each $b \geq 0$. Some intuitive properties of the model are formalized in \eqref{limitbehav}: as the operational costs increases, the optimal spread $x_v$ increases, the maximal instantaneous income $y_v$ decreases, the region where the payoff is positive gets smaller and the optimal share decreases. In particular, we remark that $\Phi_v \in \,\, ]0,1/2[$: for any value of $b$, it is never optimal to have a market share greater than $1/2$.
\end{itemize}

To sum up, we here consider the following impulsive stochastic control problem, on an infinite-horizon and with discount rate $\rho>0$.
\begin{defn}[\textbf{Admissible controls}]
	\label{def:admissiblecontrols}
	Let $x \in \rr$ be the initial state of the process and denote by $\eee$ the expectation under the condition $X_0=x$. 
	\begin{itemize}
		\item[-] A control is a sequence $u= \{(\tau_k, \delta_k)\}_{k \in \nn}$, where $\{\tau_k\}_{k \in \nn}$ are stopping times such that $0 \leq \tau_k \leq \tau_{k+1} \leq +\infty$ (the intervention times) and $\{\delta_k\}_{k \in \nn}$ are real-valued $\mathcal{F}_{\tau_k}$-measurable random variables (the corresponding impulses). 
		\item[-] A control $u= \{(\tau_k, \delta_k)\}_{k \in \nn}$ is admissible if $\lim_{t \to +\infty} \tau_k = +\infty$ and
		\begin{equation}
		\label{admisscondition}
		\eee \bigg[\sum_{k \in \nn} e^{- \rho \tau_k} \bigg] < \infty.
		\end{equation}
		We denote by $\uu_x$ the set of admissible controls.
	\end{itemize}
\end{defn}

\begin{defn}[\textbf{Value function}]
	\label{def:JandV}
	The value function $V$ is defined, for each $x \in \rr$, by
	\begin{equation}
	\label{defV}
	V(x) = \sup_{u \in \uu_x} J(x;u),
	\end{equation}
	where, for every $u= \{(\tau_k, \delta_k)\}_{k \in \nn} \in \uu_x$, we have set
	\begin{equation}
	\label{defJ}
	J(x;u) = \eee \bigg[ \int_0^\infty e^{-\rho t} R(X^{x;u}_t) dt - \sum_{k \in \nn} e^{- \rho \tau_k} c \bigg],
	\end{equation}
	with $X^{x;u}$ as in \eqref{defX} and $R$ as in \eqref{defR}. We say that $u^* \in \uu_x$  is an optimal control if $$V(x)=J(x; u^*).$$
\end{defn}

We remark that the functional $J$ in \eqref{defJ} is well-defined by \eqref{admisscondition}. Moreover, in order to simplify the notation, we will often omit the dependence on the control and the initial state, writing $X=X^{x;u}$.

\begin{remark}
	\label{rem:comm}
	
	A key-point in our model is the use of impulse controls: as underlined in the Introduction, the discrete-time nature of impulse controls better fits the practical problem we are here considering, if compared to standard continuous-time controls. On the other hand, impulse controls are technically more demanding, and getting tractable solutions sometimes requires some initial simplifications in the coefficients of the problem. Our goals are to propose an original model, provide tractable formulas with analytical existence results, show the potentialities of impulse controls for the applications in energy markets problems. Incorporating more structured elements in the model (e.g., mean-reverting $S_t$, exponential $\Phi$) is the object of an ongoing research. 
\end{remark}

\section{Verification theorem}
\label{sec:verifthm}

In this section, we provide sufficient conditions to characterize the value function and the optimal control for the problem in Section \ref{sec:formulation}. 

We briefly recall the heuristics behind the operator and the equation involved in the verification theorem below (Proposition \ref{prop:verificationONEPL}), referring the reader to \cite[Ch.~6]{OksendalSulem07} for a comprehensive introduction to stochastic impulse control theory. Given an initial state $x \in \rr$ and the function $V$ in \eqref{defV}, $V(x)$ represents the value of the problem, i.e., the maximal expected gain the retailer can achieve. Denote now by $\mm V(x)$ the value of the problem under the additional condition that the retailer immediately intervenes and behaves optimally afterwards (formal definition in \eqref{defMVonepl}). Heuristically, we have $\mm V (x) \leq V (x)$ for any state $x \in \rr$, with equality in the case where it is optimally to intervene. We then get a characterization of an optimal control: the retailer should intervene when the state is a zero of $\mm V - V$, shifting the process to the maximum point of $V$. In order to find an expression for $V$, we notice that, in the region where it is optimal to intervene, $V$ is implicitly defined by the equation $\mm V = V$. Instead, in the region where it is not optimal to intervene, the controlled process is actually a standard uncontrolled diffusion, so that the value function satisfies a suitable second-order ordinary differential equation, by a simple application of the It\^o formula. Combining the conditions in the two regions, we get the quasi-variational inequality in \eqref{defQVIonepl}. Definition \ref{def:MV} and Proposition \ref{prop:verificationONEPL} make this arguments rigorous.

\begin{defn}
	\label{def:MV}
	Let $V$ be a function from $\rr$ to $\rr$ with $\sup V \in \rr$. The function $\mm V$ is defined, for every $ x \in \rr$, by
	\begin{equation}
	\label{defMVonepl}
	\mm V (x) = \sup_{\delta \in \rr} \{V(x+\delta) - c \} = \sup V- c . 
	\end{equation}
\end{defn}

\begin{prop}[\textbf{Verification Theorem}] 
	\label{prop:verificationONEPL}
	Let the notations of Section \ref{sec:formulation} hold and let $V$ be a function from $\rr$ to $\rr$ satisfying the following conditions.
	\begin{itemize}
		\item[-] $V$ is bounded and there exists $x^* \in \rr$ such that $V(x^*)= \max_{x \in \rr} V(x)$; 
		\item[-] $D =\{ \mm V - V < 0 \}$ is a finite union of intervals;
		\item[-] $V \in C^2_b(\rr \setminus \partial D) \cap C^1_b(\rr)$;
		\item[-] $V$ is a solution to the quasi-variational inequality 
		\begin{equation}
		\label{defQVIonepl}
		\max \Big\{ \frac{\sigma^2}{2}V'' - \rho V + R, \,\, \mm V - V \Big\} =0.
		\end{equation}
	\end{itemize}
	Let $x \in \rr$ and let $u^* = \{(\tau^*_k, \delta^*_k)\}_{k \in \nn}$ be recursively defined, for $k \geq 1$, by (we omit the dependence on $x$, i.e., $u^* = u^*(x)$, $\tau^*_k = \tau^*_k(x)$, $\delta^*_k = \delta^*_k(x)$)
	\begin{gather*}
	\tau^*_{k} = \inf \big\{ t > \tau^*_{k-1} : (\mm V - V)\big(X^{x;u^*_{k-1}}_t\big)=0 \big\},
	\\
	\delta^*_{k} = x^* - X^{x;u^*_{k-1}}_{\tau^*_{k}},
	\end{gather*}
	where we have set $\tau^*_0=\delta^*_0=0$ and $u^*_k=\{(\tau^*_j, \delta^*_j)\}_{0 \leq j \leq k}$. Then, provided that $u^* \in \uu_x$,
	\begin{equation*}
	\text{$u^*$ is an optimal control and $V(x) = J(x;u^*)$.}
	\end{equation*} 
\end{prop}

\begin{proof}
	We here recall the main arguments and refer to \cite[Thm.~6.2]{OksendalSulem07} for the details. By an approximation argument, we can assume $V \in C^2_b(\rr)$. By the It\^o formula for jump processes and the dominated convergence theorem, for each $x \in \rr$ and $u \in \uu_x$ we have
	\begin{equation*}
	V(x) = \eee \bigg[ - \int_0^\infty e^{- \rho t}\Big(\frac{\sigma^2}{2}V''(X^{x;u}_t) - \rho V(X^{x;u}_t)\Big) dt + \sum_{k \in \nn} e^{-\rho \tau_k} \big(V(X^{x;u}_{\tau_k}) - V(X^{x;u}_{(\tau_k)^-})\big) \bigg].
	\end{equation*}
	By \eqref{defMVonepl} and \eqref{defQVIonepl}, we then get
	\begin{equation*}
	V(x) \geq \eee \bigg[ \int_0^\infty e^{- \rho t}R(X^{x;u}_t) dt + \sum_{k \in \nn} e^{-\rho \tau_k} c\bigg] = J(x;u).
	\end{equation*}
	In the case $u=u^*$, the same arguments lead to $V(x) = J(x;u^*)$, concluding the proof.
\end{proof}

\begin{remark}
Practically, the optimal control consists in intervening when the state exits from the region $\{\mm V - V < 0 \}$, shifting the process to the optimal state $x^*= \argmax V$. Notice that the real line is divided into two complementary region: $\{\mm V - V = 0 \}$, where the retailer intervenes (action region), and $\{\mm V - V < 0 \}$, where the retailer does not intervene (continuation region).
\end{remark}

\section{Optimal price strategy}
\label{sec:optimalcontrol}

In this section, we characterize the value function and an optimal price strategy for the problem in Section \ref{sec:formulation}, by means of the verification theorem in Section \ref{sec:verifthm}.

Our procedure is as follows. In Section \ref{ssec:candidate} we make an educated guess for the value function, solving \eqref{defQVIonepl} and imposing the regularity conditions. The candidate we get involves five parameters, implicitly defined by a system of algebraic equations. In Section \ref{ssec:welldef} we prove that a solution to this system actually exists, so that the candidate we built is well-defined. Finally, in Section \ref{ssec:applVerThm} we check that all the assumptions of Proposition \ref{prop:verificationONEPL} are satisfied and characterize an optimal price policy for the retailer.

\subsection{A candidate for the value function}
\label{ssec:candidate}

We here make an educated guess $\tilde V$ for the value function $V$. The quasi-variational inequality \eqref{defQVIonepl} suggests the following representation for $\tilde V$:
\begin{equation*}
\tilde V(x) = 
\begin{cases}
\vphi(x), & \text{if $x \in \{\mm \tilde V - \tilde V < 0 \}$,}\\
\mm\tilde V (x), & \text{if $x \in \{\mm \tilde V - \tilde V = 0 \}$,}
\end{cases}
\end{equation*}
where $\vphi$ is a solution to the equation 
\begin{equation}
\label{eqwithR}
\frac{\sigma^2}{2} \vphi'' - \rho \vphi + R =0,
\end{equation}
with $R$ as in \eqref{defR} and $\mm \tilde V$ as in \eqref{def:MV}. We now need to solve the equation in \eqref{eqwithR} and to guess an expression for $\{\mm \tilde V - \tilde V < 0 \}$ (continuation region), $\{\mm \tilde V - \tilde V = 0 \}$ (action region) and $\mm \tilde V$.

The retailer needs to intervene if the process $X$ gets too high (this implies a low number of customers, and a corresponding low income) or too small (high market share, but low unitary income). Hence, we guess that the continuation region is in the form $]\x, \xx[$, for suitable values $\x < \xx$. Moreover, we assume $]\x, \xx[$ to be included in $]0,\Delta[$. The real line is then divided into:
\begin{gather*}
\text{$ \{\mm \tilde V - \tilde V < 0 \} = \,\, ]\x, \xx[ \,\, \subseteq \,\, ]0,\Delta[$, where the retailer does not intervene,} \\
\text{$\{\mm \tilde V - \tilde V = 0 \} = \rr \setminus ]\x, \xx[ $, where the retailer intervenes.}
\end{gather*}
We now investigate the expression of $\mm \tilde V$. Heuristically, it is reasonable to assume that the candidate value function $\tilde V$ has a unique maximum point $x^*$, which belongs to the continuation region $]\x,\xx[$, where $\tilde V = \vphi$:
\begin{equation*}
\max_{y \in \rr} \{ \tilde V(y) \} \!=\! \max_{y \in ]\x,\xx[} \{ \vphi(y) \} \!=\! \vphi(x^*), 
\quad\,\,\, \text{with} \quad\,\,\, 
\vphi'(x^*) \!=\! 0, \,\,\, \vphi''(x^*) \!\leq\! 0, \,\,\, \x \!<\! x^* \!<\! \xx,
\end{equation*}
so that by \eqref{def:MV} we have
\begin{gather*}
\mm \tilde V (x) = \vphi(x^*) - c. 
\end{gather*}
Finally, recall that $R= f$ in $]\x,\xx[$, with $f$ as in \eqref{defFparab}. Then, a solution to \eqref{eqwithR} is
\begin{equation}
\label{defSolPartONEPL_Bis}
\vphi_{A_1,A_2}(x) = A_1 e^{\q x } + A_2 e^{-\q x } - k_2 (x-x_v)^2 + k_0, 
\end{equation}
for $x \in \rr$ and constants $A_1,A_2 \in \rr$, where we have set ($x_v,y_v,\alpha$ as in \eqref{defFparab})
\begin{equation}
\label{defKAPPA3}
\q = \sqrt{\frac{2 \rho}{\sigma^2}},
\quad\qquad
k_2 = \frac{\alpha}{\rho},
\quad\qquad
k_0 = \frac{y_v}{\rho} - \frac{2 k_2}{\q^2}.
\end{equation}
 The candidate $\tilde V$ depends on five parameters: $A_1,A_2, \x, \xx, x^*$. Such parameters have to be chosen so as to satisfy the regularity assumptions of the verification theorem. Namely, we need $\tilde V \in C^2(\rr \setminus \{\x,\xx\}) \cap C^1(\rr)$, so that we have to impose continuity and differentiability in $\x$ and $\xx$, leading to four equations on the five parameters of the candidate (a fifth equation comes from the optimality of $x^*$).

Summarizing the previous arguments, we get the following expression for the candidate value function $\tilde V$.
\begin{defn}
	\label{def:candidateONEPL1st}
	For every $x \in \rr$, we set
	\begin{equation*}
	\tilde V(x) = 
	\begin{cases}
	\vphi_{A_1,A_2}(x), & \text{in $]\x, \xx[$},\\
	\vphi_{A_1,A_2}(x^*) - c, & \text{in $\rr \setminus  ]\x, \xx[$},
	\end{cases}
	\end{equation*}
	where $\vphi_{A_1,A_2}$ is as in \eqref{defSolPartONEPL_Bis} and the five parameters $(A_1, A_2, \x, \xx, x^*)$ satisfy 
	\begin{equation}
	\label{ordercondition}
	0 < \x < x^* < \xx < \Delta
	\end{equation}
	and the following conditions:
	\begin{equation}
	\label{systONEPL1st}
	\begin{cases}
	\vphi_{A_1,A_2}'(x^*)=0 \,\,\text{and}\,\,\vphi_{A_1,A_2}''(x^*)<0,  & \textit{(optimality of $x^*$)}\\ 
	\vphi_{A_1,A_2}'(\x)=0,  & \textit{($C^1$-pasting in $\x$)} \\ 
	\vphi_{A_1,A_2}'(\xx)=0,  & \textit{($C^1$-pasting in $\xx$)}\\ 
	\vphi_{A_1,A_2}(\x)= \vphi_{A_1,A_2}(x^*)-c,  & \textit{($C^0$-pasting in $\x$)} \\ 
	\vphi_{A_1,A_2}(\xx)= \vphi_{A_1,A_2}(x^*)-c.  & \textit{($C^0$-pasting in $\xx$)} 
	\end{cases}
	\end{equation}
\end{defn}

\subsection{Existence and uniqueness of a solution to the coefficient system}
\label{ssec:welldef}

We now prove that Definition \ref{def:candidateONEPL1st} is well-posed, i.e., that there exists a unique solution to \eqref{ordercondition}-\eqref{systONEPL1st}. Actually, uniqueness immediately follows from the uniqueness of the value function (different solutions would imply different expressions for the value functions), so that we can just focus on the existence of solutions.

Since the underlying process is a Brownian motion and the running cost $f$  in \eqref{defFparab} is symmetric with respect to $x_v$, we expect the function $\vphi_{A_1,A_2}$ to be symmetric with respect to $x_v$, which corresponds to the choice $A_1e^{\q x_v}= A_2e^{-\q x_v}$. The same argument suggests to set $(\x+\xx)/2 = x_v$. Finally, as a symmetry point is always a local maximum or minimum point, we expect $x^* = x_v$. In short, our guess is 
\begin{equation}
\label{ourguess}
A_1 = Ae^{-\q x_v}, \qquad
A_2 = Ae^{\q x_v}, \qquad
(\x+\xx)/2 = x_v, \qquad
x^* = x_v,
\end{equation}
with $A \in \rr$. In particular, the function $\vphi_{A_1,A_2}$ in \eqref{defSolPartONEPL_Bis} now writes 
\begin{equation}
\label{defPhiA}
\vphi_A(x) = A e^{\q (x - x_v)} + A e^{- \q (x - x_v)} - k_2 (x - x_v)^2 + k_0,
\end{equation}
where $A \in \rr$ and the coefficients have been defined in \eqref{defKAPPA3}. 

An easy check shows that $x^*=x_v$ is a local maximum for $\vphi_A$, so that the first condition in \eqref{systONEPL1st} is satisfied, if and only if $A>0$. Moreover, two of the four equations left are now redundant. Then, under our guess \eqref{ourguess}, we can equivalently rewrite \eqref{systONEPL1st} as
\begin{equation*}
\begin{cases}
\vphi_A'(\xx)=0, \\ 
\vphi_A(\xx)= \vphi_A(x_v)-c,
\end{cases}
\end{equation*}
with $A>0$. As for the order condition \eqref{ordercondition}, under \eqref{ourguess} it reads $0 < \xx-x_v < \Delta - x_v$ (recall that $x_v \in [\Delta/2,\Delta[$). To simplify the notations, let
\begin{equation*}
\yy= \xx-x_v.
\end{equation*}
Then, $\tilde V$ is well-defined if there exists a solution $(A, \bar y)$ to the system 
\begin{subnumcases}{\label{sist}}
A \q e^{\q \bar y} - A \q e^{- \q \yy} - 2k_2\yy=0, \label{sistA} \\ 
A e^{\q \yy} + A e^{- \q \yy} - k_2 \yy^2 - 2A + c = 0, \label{sistB} \\
A>0, \,\, \bar y >0,
\end{subnumcases}{}
\!\!\!under the additional condition
\begin{equation}
\label{orderconditionTRASL}
\yy < \Delta -x_v.
\end{equation}

In Lemma \ref{lem:exist} we focus on \eqref{sist}, whereas in Lemma \ref{lem:ordercond} we prove that the solution to \eqref{sist} also satisfies \eqref{orderconditionTRASL}.

\begin{lemma}
	\label{lem:exist}
	There exists a unique solution $(A, \yy)$ to \eqref{sist}.
\end{lemma}

\begin{proof}
	We first consider Equation \eqref{sistA}. If we fix $A>0$, we are looking for the strictly positive zeros of the function $h_A$ defined by
	\begin{equation}
	\label{defH}
	h_A(y) = A \q e^{\q y} - A \q e^{- \q y} - 2k_2 y,
	\end{equation}
	for each $y>0$. Since
	\begin{equation*}
	h'_A(y) = A \q^2 e^{\q y} + A \q^2 e^{-\q y} - 2k_2 = \frac{A \q^2 ( e^{ \q y} )^2 -2 k_2 (e^{\q y}) + A \q^2}{e^{\q y}},
	\end{equation*}
	we separately consider the cases $A \geq \bar A$ and $0 < A < \bar A$, where 
	\begin{equation}
	\label{defAbarONEPL}
	\bar A = \frac{k_2}{\q^2} = \frac{\sigma^2(\Delta + b)}{2\rho^2 \Delta^2}.
	\end{equation}
	If $A \geq \bar A$, we have $h'_A >0$ in $]0,\infty[$. Hence, since $h_A(0)=0$, Equation \eqref{sistA} does not have any solution in $]0, +\infty[$. On the contrary, if $0< A < \bar A$, we have $h'_A<0$ in $]0,\tilde y[$ and $h'_A>0$ in $]\tilde y, \infty[$, for a suitable $\tilde y >0$. Hence, since $h_A(0)=0$ and $h_A(+\infty)=+\infty$, Equation \eqref{sistA} has exactly one solution $\bar y > \tilde y >0$. We have proved that, for a fixed $A>0$, Equation \eqref{sistA} admits a solution $\yy \in ]0,\infty[$ if and only if $A \in ]0, \bar A[$; in this case, the solution is unique and we denote it by $\yy = \yy(A)$. Finally, we remark that
	\begin{equation}
	\label{limsol}
	\lim_{A \to 0^+} \bar y(A) = +\infty, \qquad\quad 
	\lim_{A \to \bar A^-} \bar y(A) =0.
	\end{equation}
	
	We now consider Equation \eqref{sistB}. For each $A \in ]0, \bar A[$, we set
	\begin{equation}
	\label{defG}
	g(A) = -A e^{\q \yy(A)} -A e^{-\q \yy(A)} + k_2 \yy^2(A) + 2A,
	\end{equation}
	where $\bar y(A)$ is well-defined by the first step. We are going to prove that
	\begin{equation}
	\label{goals}
	\lim_{A \to 0^+} g(A) = +\infty, \qquad\qquad 
	\lim_{A \to \bar A^-} g(A) = 0, \qquad\qquad 
	g'<0.
	\end{equation}
	This concludes the proof: if \eqref{goals} holds, then Equation \eqref{sistB}, that is, $g(A)=c$, has exactly one solution $A \in ]0, \bar A[$, so that the pair $(A, \yy(A))$ is the unique solution to \eqref{sist}. 
	
	By \eqref{sistA} we can rewrite $g(A)$ as 
	\begin{equation}
	\label{defGalternative}
	g(A) = k_2 \yy^2(A) - \frac{2k_2}{\q}\frac{(e^{\q \yy(A)}-1)^2}{(e^{\q \yy(A)})^2-1}\yy(A),
	\end{equation}
	and we get the first claim in \eqref{goals} by \eqref{limsol}. The second claim in \eqref{goals} is immediate by \eqref{defG} and by \eqref{limsol}. Finally, differentiating the expression in \eqref{defG} and recalling \eqref{sistA}, we have
	\begin{equation}
	\label{defDerivG}
	g'(A) = -\frac{ (e^{\q \yy(A)} -1 )^2}{ e^{\q \yy(A)}}  - \big( A \q e^{\q \bar y} - A \q e^{-\q \yy} - 2k_2\yy \big) \yy'(A) = -\frac{ (e^{\q \yy(A)} -1 )^2}{ e^{\q \yy(A)}} <0,
	\end{equation}
	proving the third claim in \eqref{goals}.
\end{proof}

\begin{lemma}
	\label{lem:ordercond}
	Let $x_v$ be as in \eqref{defFparab} and let 
	\begin{equation*}
	\bar c = \xi(\Delta - x_v),
	\end{equation*}
	where the function $\xi$ is defined, for $y>0$, by
	\begin{equation}
	\label{defGcompA}
	\xi(y) = k_2 y^2 - \frac{2k_2}{\q}\frac{(e^{\q y}-1)^2}{(e^{\q y})^2-1}y = k_2 y^2 - \frac{2k_2}{\q}\frac{e^{\q y} + e^{-\q y}-2}{e^{\q y} - e^{-\q y}}y.
	\end{equation}
	Then, the solution $(A,\yy)$ to \eqref{sist} satisfies the order condition in \eqref{orderconditionTRASL} if and only if $c < \bar c$.
\end{lemma}

\begin{proof}	
	Let $(A,\yy)=(A, \yy(A))$ be as in the proof of Lemma \ref{lem:exist}. By \eqref{defGalternative} and the identity $g(A)=c$, we have that $\xi(\yy)= c$, with $\xi$ as in \eqref{defGcompA}. Now notice that 
	\begin{equation}
	\label{defDerivXi}
	\xi'(y) = \frac{2k_2}{\q} \frac{e^{\q y} + e^{-\q y} -2}{(e^{\q y} - e^{-\q y})^2}C(y),
	\end{equation}
	where we have set 
	\begin{equation*}
	C(y) = \q y (e^{\q y} + e^{-\q y}) - (e^{\q y} - e^{-\q y}).
	\end{equation*}
	Notice that we have $C(y)>0$ for each $y>0$, since $C(0^+)=0$ and $C'(y)>0$. Hence, the function $\xi$ is strictly increasing. As a consequence, the inequality in \eqref{orderconditionTRASL} holds, i.e., $\yy < \Delta - x_v$, if and only if $\xi(\yy) < \xi(\Delta - x_v)$, that is, if and only if $c <\bar c$, where we have used the relation $\xi(\yy)=c$ and the definition of $\bar c$.
\end{proof}

The next proposition summarizes the results proved in this section.

\begin{prop}
	\label{res:gooddefONEPL1st}
	Assume $c < \bar c$, with $\bar c$ as in Lemma \ref{lem:ordercond}. Then, the function $\tilde V$ in Definition \ref{def:candidateONEPL1st} is well-defined, as there exists a unique solution
	\begin{equation*}
	(A_{1}, A_{2}, \x, \xx,x^*)
	\end{equation*}
	to the system in \eqref{ordercondition}-\eqref{systONEPL1st}, given by
	\begin{gather*}
	A_1 = Ae^{ - \q x_v}, \qquad A_2 = Ae^{\q x_v}, 
	\\
	x^*=x_v, \qquad \x= x_v-\bar y, \qquad \xx= x_v+\bar y,
	\end{gather*}
	with $x_v$ as in \eqref{defFparab}, $\theta$ as in \eqref{defKAPPA3} and $(A,\bar y)$ the unique solution to \eqref{sist}-\eqref{orderconditionTRASL}. 
\end{prop}

\begin{remark}
Notice that the constant $\bar c$ in Lemma \ref{lem:ordercond} and Proposition \ref{res:gooddefONEPL1st} is given by 
\begin{equation*}
\bar c = \xi(\Delta - x_v) = \xi(\Delta^2/(2\Delta+2b)),
\end{equation*}
where the function $\xi$ is defined in \eqref{defGcompA} and with $x_v$ as in \eqref{defFparab}. In particular, the threshold $\bar c$ is increasing with respect to $\Delta$ and decreasing with respect to $b$. 
\end{remark}

\begin{remark}
\label{rem:Valtern}
From the proofs of this section, we notice that the function $\tilde V$ can also be represented as 
\begin{equation}
\label{defVsym}
\tilde V(x) = 
\begin{cases}
\vphi_{A}(x), & \text{in $]\x, \xx[$},\\
\vphi_{A}(x^*) - c, & \text{in $\rr \setminus  ]\x, \xx[$},
\end{cases}
\end{equation}
where $\vphi_{A}$ is defined in \eqref{defPhiA}, $x^*=x_v$ with $x_v$ as in \eqref{defFparab}, $\x + \xx = 2x_v$, $(A,\xx)$ the unique solution to
\begin{equation}
\label{sis-sym}
\begin{cases}
\vphi_A'(\xx)=0, \\ 
\vphi_A(\xx)= \vphi_A(x_v)-c.
\end{cases}
\end{equation}
More in detail, $A \in \,\, ]0, \bar A[$ is the unique solution to $g(A)=c$, with $g$ as in \eqref{defG} and $\bar A$ as in \eqref{defAbarONEPL}, and $\xx = x_v + \bar y$, where $\yy$ is the unique solution to $\xi(\yy)=c$, with $\xi$ as in \eqref{defGcompA}.
\end{remark}

\subsection{Application of the verification theorem}
\label{ssec:applVerThm}

We now check that the candidate in Section \ref{ssec:candidate}, well-defined by Section \ref{ssec:welldef}, actually satisfies all the assumptions of the verification theorem. Even if $x^*$ is explicit in our candidate, $x^*=x_v$, we prefer to keep the notation $x^*$ in the proofs of this section, in order to underline the optimality of this state.

\begin{lemma}
	\label{lem:studyofMVONEPL1st}
	Assume $c < \bar c$, with $\bar c$ as in Lemma \ref{lem:ordercond}, and let $\tilde V$ be as in Definition \ref{def:candidateONEPL1st}. Then, for every $x \in \rr$ we have	
	\begin{equation}
	\label{MVprobl}
	\mm \tilde V(x) = \vphi_{A}(x^*) - c,	
	\end{equation}
	with $A$ as in Proposition \ref{res:gooddefONEPL1st} and $\vphi_A$ as in \eqref{defPhiA}. In particular, we have
	\begin{equation}
	\label{contregONEPL1st}
	\{\mm \tilde V - \tilde V < 0 \} = \, ]\x, \xx[, \qquad
	\{\mm \tilde V - \tilde V = 0 \} = \rr \, \setminus \, ]\x, \xx[. 
	\end{equation}
\end{lemma}

\begin{proof}
	We use the representation in Remark \ref{rem:Valtern}. Notice that:
	\begin{itemize}
		\item[-] $\tilde V$ is strictly decreasing in $]x^*, \xx[$ (since we have $\tilde V = \vphi_A$ by definition and $\vphi_A'<0$ in $]x^*, \xx[$ by the proof of Lemma \ref{lem:exist});
		\item[-] $\tilde V$ is constant in $[\xx, +\infty[$ by definition of $\tilde V$, with $\tilde V \equiv \vphi_A(x^*) -c$.
	\end{itemize}	
	Since $\tilde V$ is symmetric with respect to $x^*$, we deduce that 
	\begin{equation}
	\label{maxminVtilde}
	\max_{y \in \rr} \tilde V(y) = \tilde V(x^*)=\vphi_A(x^*), 
	\qquad 
	\min_{y \in \rr} \tilde V(y) = \vphi_A(x^*)-c,
	\end{equation}	
	so that \eqref{MVprobl} holds by Definition \ref{def:MV}. Moreover, by \eqref{defVsym} it follows that 
	\begin{equation*}
	\mm \tilde V(x) - \tilde V(x) =0, \qquad \text{in $\rr \setminus ]\x, \xx[$.}
	\end{equation*}
	Finally, as $\vphi_A(\xx)= \vphi_A(x^*)-c$ by \eqref{sis-sym} and $\vphi_A(\xx) = \min_{[\x, \xx]} \vphi_A$ by the previous arguments, we have
	\begin{equation*}
	\mm \tilde V(x)  - \tilde V(x) = \vphi_A(x^*)-c - \vphi_A(x) = \vphi_A(\xx) - \vphi_A(x) <0, \qquad \text{in $]\x, \xx[$,}
	\end{equation*}
	which concludes the proof.
\end{proof}

\begin{prop}
	\label{prop:checkverifONEPL1st}
	Assume $c < \bar c$, with $\bar c$ as in Lemma \ref{lem:ordercond}, and let $\tilde V$ be as in Definition \ref{def:candidateONEPL1st}. For every $x \in \rr$, an optimal control for the problem in Section \ref{sec:formulation} is given by $u^*= \{(\tau^*_k,\delta^*_k)\}_{k\in\nn}$, where the variables $(\tau^*_k, \delta^*_k)$ are recursively defined, for $k \geq 1$, by 
	\begin{equation}
	\label{optimalcontrol}
	\begin{gathered}
	\tau^*_{k} = \inf \Big\{ t > \tau^*_{k-1} : X^{x;u^*_{k-1}}_t \notin \,\, ]\x,\xx[ \, \Big\},
	\\
	\delta^*_{k} = x^* - X^{x;u^*_{k-1}}_{\tau^*_{k}},
	\end{gathered}
	\end{equation}
	where we have set $\tau^*_0=\delta^*_0=0$ and $u^*_k=\{(\tau^*_j, \delta^*_j)\}_{0 \leq j \leq k}$.  Moreover, $\tilde V$ coincides with the value function: for every $x \in \rr$ we have
	\begin{equation*}
	\tilde V(x) = V(x) = J(x;u^*).
	\end{equation*} 
\end{prop}

\begin{remark}
	Practically, an optimal price policy for the retailer consists in intervening when $X$ exits from $]\x,\xx[$. When this happens, the retailer shifts the process to the state $x^*=x_v$. The parameters $\x,\xx$ are defined by an algebraic system of equations, for which existence and uniqueness results have been provided in Section \ref{ssec:welldef}.
\end{remark}

\begin{proof}
	We have to check that the candidate $\tilde V$  satisfies all the assumptions of Proposition \ref{prop:verificationONEPL}. For the reader's convenience, we briefly report the conditions to check:
	\begin{itemize}
		\item[(i)] $\tilde V$ is bounded and $\max_{x \in \rr} \tilde V(x)$ exists;
		\item[(ii)] $\tilde V \in C^2_b(\rr \setminus \{ \x, \xx \}) \cap C^1_b(\rr)$;
		\item[(iii)] $\tilde V$ satisfies $\max \{ \aaa \tilde V - \rho \tilde V + R, \mm \tilde V - \tilde V \} =0$;
		\item[(iv)] the optimal control is admissible, i.e., $u^* \in \uu_x$ for every $x \in \rr$.
	\end{itemize}
	
	\textit{Condition (i) and (ii).} The first condition holds by \eqref{maxminVtilde}, whereas the second condition follows by the definition of $\tilde V$.
	
	\textit{Condition (iii).} We have to prove that for every $x \in \rr$ we have
	\begin{equation}
	\max \{ \aaa\tilde V(x) - \rho \tilde V(x) + R(x), \mm \tilde V(x) - \tilde V(x) \} =0.
	\end{equation}
	We use the representation in Remark \ref{rem:Valtern}. In $]\x, \xx[$ the claim is true, as $\mm \tilde V- \tilde V < 0$ by \eqref{contregONEPL1st} and 
	\begin{equation*}
	\frac{\sigma^2}{2} \tilde V'' -\rho \tilde V+ R = \frac{\sigma^2}{2} \vphi_A'' - \rho \vphi_A + f= 0,
	\end{equation*}
	by the definition of $\vphi_A$. As for $\rr \setminus ]\x, \xx[$, we already know by \eqref{contregONEPL1st} that $\mm \tilde V- \tilde V= 0$. Then, to conclude we have to prove that
	\begin{equation*}
	\frac{\sigma^2}{2}  \tilde V''(x) -\rho \tilde V(x)+ R(x) \leq 0, \qquad \text{$\forall x \in \rr \setminus ]\x, \xx[$.}
	\end{equation*}
	By symmetry, it is enough to prove the claim for $x \in [\xx,+\infty[$. By the definition of $\tilde V(x)$ and \eqref{sis-sym}, in the interval $[\xx,+\infty[$ we have $\tilde V \equiv \vphi_A(x^*) - c = \vphi_A(\xx)$; hence, the inequality reads
	\begin{equation*}
	-\rho \vphi_A(\xx) + R(x) \leq 0, \qquad \text{$\forall x \in [\xx,+\infty[$.}
	\end{equation*}
	As $R$ is decreasing in $[ x^*, +\infty[ \,\, \supseteq [\xx,+\infty[$, it is enough to prove the claim in $x=\xx$:
	\begin{equation*}
	-\rho \vphi_A(\xx) + R(\xx) \leq 0.
	\end{equation*}
	Since $(\sigma^2/2) \vphi_A'' (\xx) - \rho \vphi_A(\xx) + f(\xx) =0$ and $f(\xx)=R(\xx)$, we can rewrite as
	\begin{equation*}
	-\frac{\sigma^2}{2} \vphi_A''(\xx) \leq 0,
	\end{equation*}
	which is true as $\xx$ is a local minimum of $\vphi_A \in C^\infty(\rr)$.
	
	\textit{Condition (iv).} Given $x \in \rr$, by Definition \ref{def:admissiblecontrols} we have to show that
	\begin{equation*}
	\eee \bigg[\sum_{k\geq 1} e^{- \rho \tau^*_k} \bigg] < \infty.
	\end{equation*}
	When acting according to the optimal control $u^*$, the retailer intervenes when the process hits $\x$ or $\xx$ and shifts the process to $x^* \in \,\,]\x, \xx[$. As a consequence, we can decompose each variable $\tau^*_k$ as a sum of suitable exit times from $]\x, \xx[$. Given $y \in \rr$, let $\zeta^y$ denote the exit time of the process $y + \sigma W$, where $W$ is a real Brownian motion, from the interval $]\x, \xx[$; then, we have $\tau^*_1= \zeta^x$ and
	\begin{equation*}
	\tau^*_k = \zeta^x + \sum_{l=1}^{k-1} \zeta^{x^*}_l,
	\end{equation*}
	for every $k \geq 2$, where the variables $\zeta^{x^*}_l$ are independent and distributed as $\zeta^{x^*}$. As a consequence, we have
	\begin{equation*}
	\eee \bigg[\sum_{k \geq 2} e^{- \rho \tau^*_k} \bigg]
	= \eee \bigg[\sum_{k \geq 2} e^{- \rho \big(\zeta^x + \sum_{l=1}^{k-1} \zeta^{x^*}_l\big)} \bigg] 
	= \eee \bigg[ e^{-\rho \zeta^x}\sum_{k \geq 2} \prod_{l=1,\dots,k-1} e^{-\rho \zeta^{x^*}_l} \bigg].
	\end{equation*}
	By the Fubini-Tonelli theorem and the independence of the variables:
	\begin{equation*}
	\eee \bigg[ e^{-\rho \zeta^x}\sum_{k \geq 2} \prod_{l=1,\dots,k-1} e^{-\rho \zeta^{x^*}_l} \bigg]
	= \eee \big[ e^{-\rho \zeta^x}\big]\sum_{k \geq 2} \,\, \prod_{l=1,\dots,k-1} \eee \Big[ e^{-\rho \zeta^{x^*}_l} \Big]. 
	\end{equation*}
	As the variables $\zeta^{x^*}_l$ are identically distributed with $\zeta^{x^*}_l \sim \zeta^{x^*}$, we can conclude:
	\begin{equation*}
	\sum_{k \geq 2} \prod_{l=1,\dots,k-1} \eee \Big[ e^{-\rho \zeta^{x^*}_l} \Big]
	=\sum_{k \geq 2} \eee \Big[ e^{-\rho \zeta^{x^*}} \Big]^{k-1} < \infty,
	\end{equation*}
	which is a converging geometric series.
\end{proof}

\paragraph{Numerical simulations.} We conclude this section with some numerical simulations, obtained with Wolfram Mathematica. We consider the following two sets of parameters and plot the value functions in the interval $[0,\Delta]$.
\begin{align*}
&\text{Problem 1: $\,\, \rho=0.03, \,\,\, \sigma=0.2, \,\,\, b=0.0, \,\,\, c=2.0, \,\,\, \Delta=5.0$,} \\
&\text{Problem 2: $\,\, \rho=0.05, \,\,\, \sigma=0.3, \,\,\, b=3.0, \,\,\, c=0.5, \,\,\, \Delta=5.0$.}
\end{align*}
We also plot, as a dashed line, the function $\vphi_{A}$. We notice the $C^1$-pasting in $\x,\xx$ and the maximum point in $x^*=x_v$. Also, as remarked in Section \ref{sec:formulation}, the coefficients $b$ moves the maximum point towards the boundary $\Delta$. Finally, the parameter $c$ has clearly an impact on the size of the interval $]\x,\xx[$: we investigate this property in the next section.

\begin{figure}[h!]
	\begin{minipage}{0.45\textwidth}
		\centering
		\includegraphics[width=\textwidth]{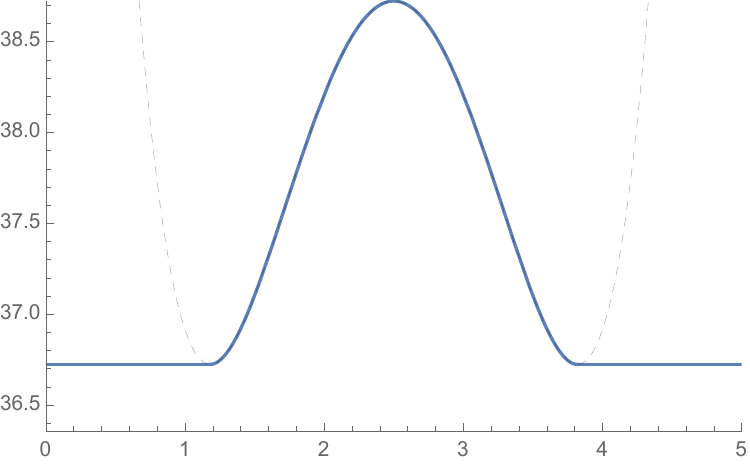}
		\caption{\footnotesize{$x \mapsto V(x)$ for Problem 1.}}
		\label{fig:0PICsym1}
	\end{minipage}
	\hfill 
	\begin{minipage}{0.45\textwidth}
		\centering
		\includegraphics[width=\textwidth]{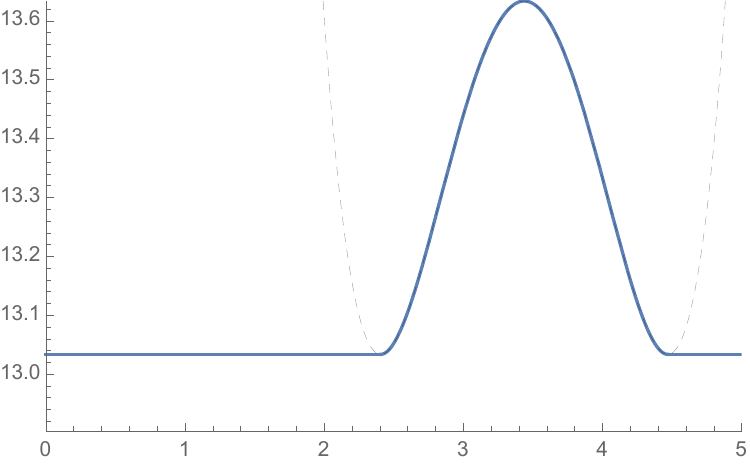}
		\caption{\footnotesize{$x \mapsto V(x)$ for Problem 2.}}
		\label{fig:0PICsym2}
	\end{minipage}
\end{figure}

\section{Estimates with respect to the intervention cost}
\label{sec:estimates}

In this section, we investigate the impact of the parameter $c>0$ (intervention cost) on the value function $V$ and on the continuation region $]\x, \xx[$. The main result of this section is an asymptotic estimate for $]\x,\xx[$ as $c \to 0^+$, see Proposition \ref{prop:asympONEPL1st}.

We briefly recall from Section \ref{ssec:welldef} the symmetric representation for the value function (we here write $V=V^{c}$, $\x=\x(c)$, $\xx=\xx(c)$, $A=A(c)$, in order to stress the dependence on $c$):
\begin{equation}	
\label{math9bisONEPL}
V^c(x) = 
\begin{cases}
\vphi_{A(c)}(x), & \text{$x \in \,\, ]\x(c), \xx(c)[$},\\
\vphi_{A(c)}(x_v) - c, & \text{$x \in \rr \setminus  ]\x(c), \xx(c)[$},
\end{cases}
\end{equation}
where the function $\vphi_{A(c)}$ is defined in \eqref{defPhiA} and $\x(c),\xx(c)$ are defined by
\begin{equation}
\label{math0ONEPL}
\x(c)= x_v - \yy(c),
\qquad
\xx(c)= x_v + \yy(c), 		
\end{equation}
with $A(c) \in \,\, ]0, \bar A[$ and $\yy(c)>0$ implicitly defined by the equations
\begin{equation}
\label{defAcRecall}
g(A(c))=c, \qquad\qquad \xi(\yy(c))=c,
\end{equation}
with $g$ as in \eqref{defG} and $\xi$ as in \eqref{defGcompA}. Recall that we have $\xi'>0$ and $g'<0$, so that $A(c)$ and $\yy(c)$ are well-defined, we refer to Section \ref{ssec:welldef} for the details.

\paragraph{Asymptotic estimates as $c \to 0^+$.} 

If the intervention costs vanish, i.e., $c \to 0^+$, intervening gets costless, so that we expect the retailer to continuously intervene to keep the process in the optimal state. Hence, $X_t^{x;u^*(c)}$ gets constant and equal to $x^*=x_v$, the value function $V^c$ converges to $\int_0^\infty e^{-\rho t} R(x_v) dt = y_v/\rho$ with $y_v$ as in \eqref{defFparab}, the continuation region $]\x(c),\xx(c)[$ collapses into the singleton $\{x_v\}$. Notice that the limit optimal control is formally not admissible, as it would require continuous-time interventions.

Actually, thanks to the results in the previous section, we can prove a stronger result. Namely, we investigate the asymptotic behaviour of $]\x(c),\xx(c)[$ as $c \to 0^+$, proving that $\x(c),\xx(c)$ converge to $x_v$ like the fourth root of $c$. To our knowledge, this is the first time that an asymptotic estimate for the continuation region of an impulse control problem is provided.

\begin{prop}
	\label{prop:asympONEPL1st}
	The following asymptotic estimates hold: 
	\begin{equation}
	\label{asymptotic}	
	\x(c) \sim_{c \to 0^+} x_v - C \sqrt[4]{c}, 
	\qquad
	\xx(c) \sim_{c \to 0^+} x_v + C \sqrt[4]{c}, 
	\end{equation}	
	where we have set $C = \sqrt[4]{6  \sigma^2 / \alpha}$, with $\alpha$ as in \eqref{defFparab}. In particular, we have
	\begin{equation}
	\label{math3ONEPLbis}
	\lim_{c \to 0^+} \x(c) = \lim_{c \to 0^+} \xx(c)= x_v.
	\end{equation}
\end{prop}

\begin{proof}
	By \eqref{math0ONEPL} we have to estimate $\yy(c)$ as $c \to 0^+$. Since $\yy(c)$ is defined by $\xi(\yy(c))=c$, see \eqref{defAcRecall}, we start by estimating the function $\xi$. Recall by \eqref{defGcompA} that
	\begin{equation*}
	\xi(y) = \frac{\theta k_2 y^2(e^{\q y} - e^{-\q y}) -2k_2 y(e^{\q y} + e^{-\q y} -2\big)}{\theta (e^{\q y} - e^{-\q y})},
	\end{equation*}
	for every $y >0$. By the Taylor series we have
	\begin{align*}
	&e^{\q y}- e^{-\q y} = 2\q y + \frac{1}{3} \q^3 y^3 + o\big( y^4\big), \\
	&e^{\q y}+ e^{-\q y} = 2 + \q^2 y^2 + \frac{1}{12} \q^4 y^4 + o\big(y^5\big), 
	\end{align*}
	leading to the following approximation:
	\begin{equation}
	\label{math1ONEPL} 
	\xi(y) \sim_{y \to 0^+} \frac{\theta k_2 y^2\big(2 \q y + \q^3 y^3/3\big) -2k_2 y\big(\q^2y^2 + \q^4 y^4/12\big) }{2 \q^2 y}=\frac{k_2 \q^2}{12} y^4.
	\end{equation}
	Since $\xi(0^+)=0$ by \eqref{defGcompA} and $\xi$ is a one-to-one map by \eqref{defDerivXi}, from the relation in \eqref{defAcRecall} we deduce that 
	\begin{equation}
	\label{math2ONEPL}
	\yy(0^+) = 0.
	\end{equation}
	It then follows from \eqref{math1ONEPL} that
	\begin{equation*}
	\xi(\yy(c)) \sim_{c \to 0^+} \frac{k_2 \q^2}{12} \yy^4(c).
	\end{equation*}
	The relation $\xi(\yy(c))=c$, see \eqref{defAcRecall}, then implies
	\begin{equation}
	\label{defApproxYAC}
	\yy(c) \sim_{c \to 0^+} \sqrt[4]{\frac{12}{k_2 \q^2}} \, \sqrt[4]{c},
	\end{equation}
	which concludes the proof, since $k_2 \q^2 = 2\alpha/\sigma^2$.
\end{proof}

\begin{remark}
	The coefficient $C$ is increasing w.r.t.~$\sigma$ and decreasing w.r.t.~$\alpha$, which is reasonable: if the volatility $\sigma$ increases, the continuation region gets bigger to reduce the frequency of the interventions, whereas, if the concavity $\alpha$ of the payoff increases, the continuation region gets smaller to keep the process close to the optimal value. Finally, we notice that the asymptotic estimate in \eqref{asymptotic} is independent of the discount factor $\rho$.
\end{remark}

\begin{remark}
	\label{rem:extest}
	We remark that \eqref{asymptotic} holds for a more general class of problems, namely, for any control problem with symmetric quadratic payoff, constant intervention costs and Brownian underlying. More in detail, for any problem in the form 
	\begin{equation*}
	V(x)=\sup_{u \in \uu_x} \eee \bigg[ \int_0^\infty e^{-\rho t} (X^{x;u}_t - x_v)^2 dt - \sum_{k \in \nn} e^{- \rho \tau_k} c \bigg], \qquad\,\,\,
	X^{x;u}_t = x + \sigma W_t + \sum_{ \tau_k \leq t} \delta_k,
	\end{equation*}
	the continuation region is in the form $]\x(c),\xx(c)[$ and converges to the singleton $\{x_v\}$ as $c \to 0^+$, with estimate
	\begin{equation*}
	\x(c),\xx(c) \sim_{c \to 0^+} x_v \pm \tilde C \sqrt[4]{c}, \qquad\qquad \tilde C = \sqrt[4]{6 \sigma^2}.
	\end{equation*}
	To our knowledge, there are no references in the literature for similar estimates.
\end{remark}

\begin{prop}
	\label{prop:limitONEPL1st}
	The following pointwise limits hold:

	\begin{equation}
	\label{math6ONEPL}
	\lim_{c \to 0^+} V^c(x)= V^{\text{static}}, \qquad
	\lim_{c \to 0^+} X_t^{x; u^*(c)} = x_v,
	\end{equation}
	for every $x \in \rr$ and $t \geq 0$, where we have set $V^{\text{static}} = y_v/\rho$, with $y_v$ as in \eqref{defFparab}.
\end{prop}

\begin{proof}
	Since $\x(c)$ and $\xx(c)$ converge to $x_v$, for any $x \in \rr \setminus \{x_v\}$ we have that $x \in \rr \setminus ]\x(c),\xx(c)[$ for $c$ small enough (say, $c \leq \tilde c(x)$). Hence, by \eqref{math9bisONEPL} we have
	\begin{equation}
	\label{math10ONEPL}
	V^c(x) = \vphi_{A(c)}(x_v) - c, \qquad c \in \,\,]0, \tilde c(x)[, \,\,\, x \in \rr \setminus \{x_v\}.
	\end{equation}
	Instead, in the case $x=x_v$, since $x_v \in ]\x(c),\xx(c)[$, by \eqref{math9bisONEPL} we have
	\begin{equation}
	\label{math11ONEPL}
	V^c(x_v)= \vphi_{A(c)}(x_v), \qquad c \in \,\, ]0,+\infty[.	
	\end{equation}
	By \eqref{math10ONEPL} and \eqref{math11ONEPL} it follows that	for every $x \in \rr$ we have
	\begin{equation*}
	\lim_{c \to 0^+} V^c(x)= \lim_{c \to 0^+} \vphi_{A(c)}(x_v) = \vphi_{A(0^+)}(x_v).
	\end{equation*}
	Let $\bar A$ be as in \eqref{defAbarONEPL}. Since $g(\bar A^-)=0$ by \eqref{defG} and since $g$ is a one-to-one map by \eqref{goals}, the identity in \eqref{defAcRecall} implies that 
	\begin{equation*}
	A(0^+) = \bar A.
	\end{equation*}
	Then, by the definition of $\vphi_{\bar A}$ in \eqref{defPhiA}, the definition of $\bar A$ in \eqref{defAbarONEPL} and the value of $k_0$ in \eqref{defKAPPA3}, we have
	\begin{equation*}
	\lim_{c \to 0^+} V^c(x)=\vphi_{\bar A}(x_v) = 2\bar A + k_0 =  \frac{y_v}{\rho},
	\end{equation*}
	which proves the first claim in \eqref{math6ONEPL}. Finally, since
	\begin{equation*}
	\x(c) < X^{x; u^*(x,c)} < \xx(c)
	\end{equation*}
	by the definition of $u^*(c)$, the second claim in \eqref{math6ONEPL} immediately follows by passing to the limit as $c \to 0^+$.
\end{proof}

\paragraph{Monotonicity.} We now investigate the monotonicity of the continuation region and the value function with respect to $c$. When the intervention cost decreases, the retailer intervenes more frequently, so that we expect a smaller continuation region and a bigger value for the problem. 
	
\begin{prop}
	\label{prop:limitONEPL1stbis}
	For every $c>0$ and $x \in \rr$ we have
	\begin{equation}
	\label{math3ONEPL}	
	\x'(c)<0, 
	\qquad
	\xx'(c)>0,
	\qquad 
	\frac{d}{dc} V^{c}(x) <0.
	\end{equation}
	In particular, the value function is always smaller than the static maximum:
	\begin{equation}
	\label{math5ONEPL}
	V^{c}(x) < V^{\text{static}},
	\end{equation}
	where $V^{\text{static}}$ is the constant defined in Proposition \ref{prop:limitONEPL1st}.
\end{prop}

\begin{proof}
	By \eqref{math0ONEPL} and \eqref{defAcRecall}, we have
	\begin{equation*}
	\xx'(c)=\yy'(c)=\frac{1}{\xi'(\yy(c))} >0,
	\end{equation*}
	where we recall that $\xi'>0$ by \eqref{defDerivXi}. By symmetry, it then follows that $\x'(c)<0$. As for the value function, by \eqref{math9bisONEPL} and the definition of $\vphi_{A(c)}$ in \eqref{defPhiA} we have
	\begin{equation}
	\label{math12ONEPL}
	\frac{dV^c}{dc}(x) = 
	\begin{cases}
	A'(c)\big(e^{\q (x - x_v)} + e^{- \q (x - x_v)}\big) , & \text{$x \in \,\, ]\x(c),\xx(c)[$}, \\
	2A'(c)-1, & \text{$x \in \,\, \rr \setminus ]\x(c),\xx(c)[$},
	\end{cases}
	\end{equation}
	which concludes the proof since by \eqref{defAcRecall} we have
	\begin{equation}
	\label{derivAc}
	A'(c)=\frac{1}{g'(A(c))} < 0,
	\end{equation}
	where we recall that $g'<0$ by \eqref{goals}. 
\end{proof}

\paragraph{Robustness.} Finally, we study how sensitive the value function is with respect to small changes in $c$ around zero. First in \cite{Oksendal99} and then in \cite{OksendalUboeZhang07}, it has been shown that intervention costs in the form $c + \tilde\lambda|\delta|$, with $\tilde \lambda>0$, imply that $dV^c/dc$ diverges as $c \to 0^+$, which is extremely problematic when performing numerical experiments. Our problem does not belong to the class studied in \cite{OksendalUboeZhang07}, as we here have $\tilde\lambda=0$. However, this property is present in our case as well. Thanks to the estimate in the previous sections, we can actually prove a stronger result, providing an asymptotic estimate for the derivative as $c \to 0^+$.

\begin{prop}
	\label{corol:robustnessONEPL1st}
	For every $x \in \rr$, we have
	\begin{equation*}
	\frac{dV^c}{dc} (x) \sim_{c \to 0^+} - \hat C \frac{1}{\sqrt{c}},
	\end{equation*}
	where $\hat C = \sigma / (\sqrt{6} \rho)$. In particular, we have
	\begin{equation*}
	\lim_{c \to 0^+} \frac{dV^c}{dc} (x) = -\infty. 
	\end{equation*}
\end{prop}

\begin{proof}
	By \eqref{derivAc} and \eqref{defDerivG} we have
	\begin{equation*}
	A'(c) =\frac{1}{g'(A(c))} = -\frac{ e^{\q \yy(c)} }{ (e^{\q \yy(c)} -1 )^2 },
	\end{equation*}
	for any $c>0$. By the Taylor series and \eqref{defApproxYAC}, we then have
	\begin{equation*}
	A'(c) \sim_{c \to 0^+} -\frac{1}{\q^2 \yy^2(c)} \sim_{c \to 0^+} -\frac{1}{\q^2 C^2 \sqrt{c}},
	\end{equation*}
	with $C$ as in Proposition \ref{prop:asympONEPL1st}. In particular, $A'(0^+)=-\infty$, so that by \eqref{math12ONEPL} we have (consider separately $x \in \rr \setminus \{x_v\}$ and $x = x_v$, recalling that $\x(c),\xx(c)$ converge to $x_v$)
	\begin{equation*}
	\frac{dV^c}{dc}(x) \sim_{c \to 0^+} 2 A'(c) \sim_{c \to 0^+} -\frac{2}{\q^2 C^2 \sqrt{c}}. \qedhere
	\end{equation*}
\end{proof}

\paragraph{Numerical simulations.} We conclude with some numerical simulations showing the results above.  Figure \ref{fig:0PICasymp1} considers the data of Problem 1 (see Section \ref{ssec:applVerThm}) and plots the function $x \mapsto V^c(x)$ for the values (from top to bottom) $c=0^+, \, 1, \, 5, \, 10, \, 15$. We here have $\bar c = 16.9$, so all such values are admissible. We notice that $V^{0^+} \equiv V^{static}$, that the interval $]\x,\xx[$ increases as $c$ increases and that the value $V^c(x)$ decreases as $c$ increases, for any $x \in \rr$. Figure \ref{fig:0PICasymp2} also refers to Problem 1 and plots the boundaries of the continuation region as a function of $c$, i.e., $c \mapsto \x(c), \xx(c)$. We also plot, as dashed lines, the optimal state $x^*=x_v$ and the asymptotic estimates $c \mapsto x_v \pm C\sqrt[4]{c}$, with $C$ as in Proposition \ref{prop:asympONEPL1st}. We notice the monotonicity properties and the limit as $c \to 0^+$, when the continuation region degenerates into the singleton $\{x_v\}$.

\begin{figure}[h!]
	\begin{minipage}{0.45\textwidth}
		\centering
		\includegraphics[width=\textwidth]{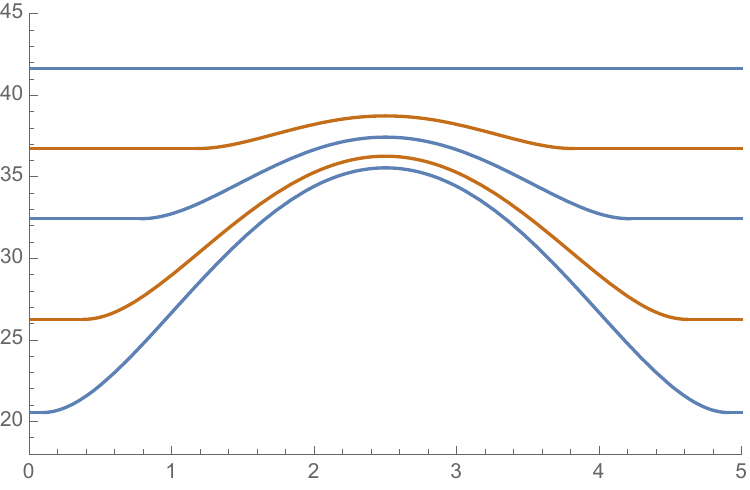}
		\caption{\footnotesize{$x \mapsto V^c(x)$ for Problem 1 and increasing values of $c$.}}
		\label{fig:0PICasymp1}
	\end{minipage}
	\hfill 
	\begin{minipage}{0.45\textwidth}
		\centering
		\includegraphics[width=\textwidth]{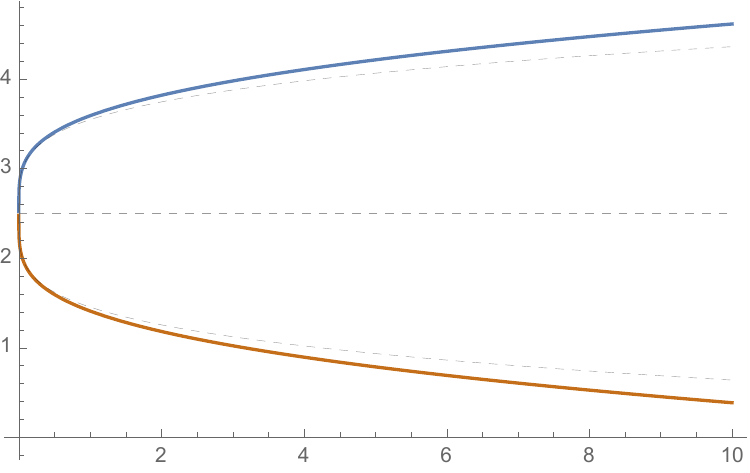}
		\caption{\footnotesize{$c \mapsto x(c),\xx(c)$ for Problem 1.}}
		\label{fig:0PICasymp2}
	\end{minipage}
\end{figure}

\section{Extensions of the model}
\label{sec:extensions}

In this section, we extend the model in Section \ref{sec:formulation} and provide suitable adaptations of the previous results. We consider the following extensions.
\begin{itemize}
	\item[-] We add a drift term to the wholesale price: $S_t = s + \mu t + \sigma W_t$, for a constant $\mu \in \rr$. As a consequence, given $u = \{ (\tau_k, \delta_k) \}_{k \in \nn}$, the controlled process is now
	\begin{equation}
	\label{defX-new}
	X_t^{x;u} = P_t - S_t = x - \mu t - \sigma W_t + \sum_{\tau_k \leq t} \delta_k.
	\end{equation}
	
	\item[-] In Section \ref{sec:formulation}, the intervention cost was a fixed constant $c$. We now add a state-dependent term, proportional to the market share $\Phi(X_t^{x;u})$. Namely, we assume that the cost is given by $K(X_t^{x;u})$, with $X_t^{x;u}$ as in \eqref{defX-new} and where
	\begin{equation}
	\label{defK-new}
	K(x) = c + \lambda \Phi(x)=
	\begin{cases}
	\lambda + c, & \text{if $x<0$,}\\
	- \frac{\lambda}{\Delta} (x - \Delta) + c, & \text{if $0 \leq x \leq \Delta$,}\\
	c, & \text{if $x > \Delta$,}
	\end{cases}
	\end{equation}
	where $x \in \rr$ is the state of the process before the intervention and $c>0, \lambda \geq 0$ are fixed constants. For $\lambda=0$, we retrieve the constant costs of Section \ref{sec:formulation}. Correspondingly, the functional $J$ in \eqref{defJ} now reads (with $R$ as in \eqref{defR})
	\begin{equation}
	\label{defJ-new}
	J(x;u) = \eee \bigg[ \int_0^\infty e^{-\rho t} R(X^{x;u}_t) dt - \sum_{k \in \nn} e^{- \rho \tau_k} K\big(X^{x;u}_{(\tau_k)^-}\big) \bigg].
	\end{equation}
	
	\item[-] The cost function in \eqref{defK-new} is not smooth, and we cannot apply the verification theorem. However, if the controlled process $X$ never exits from the interval $]0,\Delta[$, the singularities of the penalty function no longer belong to the set where the value function is defined, and the verification theorem can be applied. Then, we require a further condition to admissible controls $u \in \uu_x$, besides the one in Definition \ref{def:admissiblecontrols}:
	\begin{equation}
	\label{finallyawayout}
	\eee \bigg[\sum_{k \in \nn} e^{- \rho \tau_k} K\big(X^{x;u}_{(\tau_k)^-}\big)\bigg] < \infty,
	\qquad\qquad
	X^{x;u}_t \in \,\,]0,\Delta[, \quad \forall t \geq 0.
	\end{equation}
	Practically, \eqref{finallyawayout} forces the retailer to intervene (at least) every time his market share hits $0$ or $1$; in other words, we do not admit situations where the retailer has no customers or where he holds the monopoly of the market, making this assumption mild and reasonable from a practical point of view.
\end{itemize}

The verification theorem in Proposition \ref{prop:verificationONEPL} still holds in this new framework, with minor changes. In particular, the quasi-variational inequality \eqref{defQVIonepl} now writes
\begin{equation}
\label{defQVIonepl-new}
\max \Big\{ \frac{\sigma^2}{2} V'' - \mu V' - \rho V + R, \,\, \mm V - V \Big\} =0.
\end{equation}
The solution to $(\sigma^2/2)V'' - \mu V' - \rho V + f =0$ is given by
\begin{equation}
\label{defSolPartONEPL}
\vphi_{A_1,A_2}(x) = A_1 e^{m_1 x } + A_2 e^{m_2 x } - k_2 x^2 + k_1 x - \tilde k_0, 
\end{equation}
where $A_1,A_2 \in \rr$ and we have set
\begin{equation}
\label{defCoeffSolPartONEPL}
\begin{gathered}
m_{1,2} = \frac{\mu \pm  \sqrt{\mu^2 + 2 \rho \sigma^2}}{\sigma^2}, \\
k_2 = \frac{\alpha}{\rho}, \qquad
k_1 = \frac{2 \alpha x_v}{\rho} + \frac{2\alpha \mu}{\rho^2}, \qquad
\tilde k_0 = \frac{\alpha x_v^2 - y_v}{\rho} + \frac{\alpha (\sigma^2 + 2 \mu x_v)}{\rho^2} + \frac{2\alpha\mu^2}{\rho^3},
\end{gathered}
\end{equation}
with $\alpha, x_v, y_v$ as in \eqref{defFparab}. Given the equation in \eqref{defQVIonepl-new}, the same arguments as the ones in Section \ref{ssec:candidate} lead to the following candidate value function.

\begin{defn}
	\label{def:candidateGeneralCaseBIS}
	For each $x \in \,\, ]0,\Delta[$, we set
	\begin{gather*}
	\tilde V(x) = 
	\begin{cases}
	\vphi_{A_1,A_2}(x), & \text{in $]\x,\xx[$},\\
	\vphi_{A_1,A_2}(x^*) - c + \lambda/\Delta(x-\Delta), & \text{in $]0,\Delta[ \setminus ]\x,\xx[$},
	\end{cases}	
	\end{gather*}
	where $\vphi_{A_1,A_2}$ is as in \eqref{defSolPartONEPL} and the five parameters $(A_1, A_2, \x, \xx, x^*)$ satisfy
	\begin{equation}
	\label{ordercondGeneral}
	0 < \x < x^* < \xx < \Delta 
	\end{equation}
	and the following conditions:
	\begin{equation}
	\label{systGeneral}
	\begin{cases}
	\vphi_{A_1,A_2}'(x^*)=0 \,\,\,\,\,\text{and}\,\,\,\,\,\vphi_{A_1,A_2}''(x^*)<0,  & \textit{(optimality of $x^*$)}\\ 
	\vphi_{A_1,A_2}'(\x)=\lambda/\Delta,  & \textit{($C^1$-pasting in $\x$)} \\ 
	\vphi_{A_1,A_2}'(\xx)=\lambda/\Delta,  & \textit{($C^1$-pasting in $\xx$)}\\ 
	\vphi_{A_1,A_2}(\x)= \vphi_{A_1,A_2}(x^*)-c + \lambda/\Delta(\x-\Delta),  & \textit{($C^0$-pasting in $\x$)} \\ 
	\vphi_{A_1,A_2}(\xx)= \vphi_{A_1,A_2}(x^*)-c+ \lambda/\Delta(\xx-\Delta).  & \textit{($C^0$-pasting in $\xx$)} 
	\end{cases}
	\end{equation}
\end{defn}

The new structure of the problem does not allow an easy adaptation of the techniques in Section \ref{ssec:welldef} and we need to rely on numerical simulations to verify the existence of a solution to \eqref{ordercondGeneral}-\eqref{systGeneral}. Provided that such a solution exists, the verification theorem applies and we get the following characterization of an optimal control.

\begin{prop}
	Assume that a solution to \eqref{ordercondGeneral}-\eqref{systGeneral} exists. Moreover,  assume that there exist $\tilde x_1, \tilde x_2$, with $\x <\tilde x_1 < x^*<\tilde x_2 <\xx$, such that $\vphi_{A_1,A_2}''<0$ in $]\tilde x_1, \tilde x_2[$ and  $\vphi_ {A_1,A_2}''>0$ in $]\x,\tilde x_1[ \,\cup\, ]\tilde x_2, \xx[$. Finally, assume $\x < \hat x <\xx$, where we have set $\hat x = x_v - (\rho \lambda) / (2 \alpha \Delta)$.
	
	Then, for every $x \in ]0,\Delta[$ an optimal control $u^*$ for the problem described in this section exists and is given by \eqref{optimalcontrol}. Moreover, the value function of the problem is given by $\tilde V$ in Definition \ref{def:candidateGeneralCaseBIS}.
\end{prop}

\begin{proof}
	Easy adaptation of the proofs in Section \ref{ssec:applVerThm}.
\end{proof}

\begin{remark}
The practical interpretation in Section \ref{ssec:applVerThm} still holds: the retailer should intervene when the state variable exits from the region $]\x,\xx[$, shifting the process to the optimal state $x^*$. The difference is that the system characterizing $\x,\xx,x^*$ is no longer analytically tractable.
\end{remark}

\paragraph{Numerical simulations.} We here consider the following sets of parameters and plot the corresponding value functions.
\begin{align*}
&\text{Problem 3: \, $\rho=0.03, \,\,\, \mu=0.2, \,\,\, \sigma=0.25, \,\,\, b=0.4, \,\,\, c=1.5, \,\,\, \lambda=0.0, \,\,\, \Delta=5.0$,} \\
&\text{Problem 4: \, $\rho=0.05, \,\,\, \mu=0.1, \,\,\, \sigma=0.3, \,\,\,\,\,\, b=0.5, \,\,\, c=1.0, \,\,\, \lambda=0.5, \,\,\, \Delta=5.0$.}
\end{align*}
We also plot, as a dashed line, the function $\vphi_{A_1,A_2}$: notice the $C^1$-pasting in $\x,\xx$. As expected, the value functions are no longer symmetric: the right part of the bell-shaped curve is steeper that the left one, due to the drift in the underlying process.

\begin{figure}[h!]
	\begin{minipage}{0.45\textwidth}
		\centering
		\includegraphics[width=\textwidth]{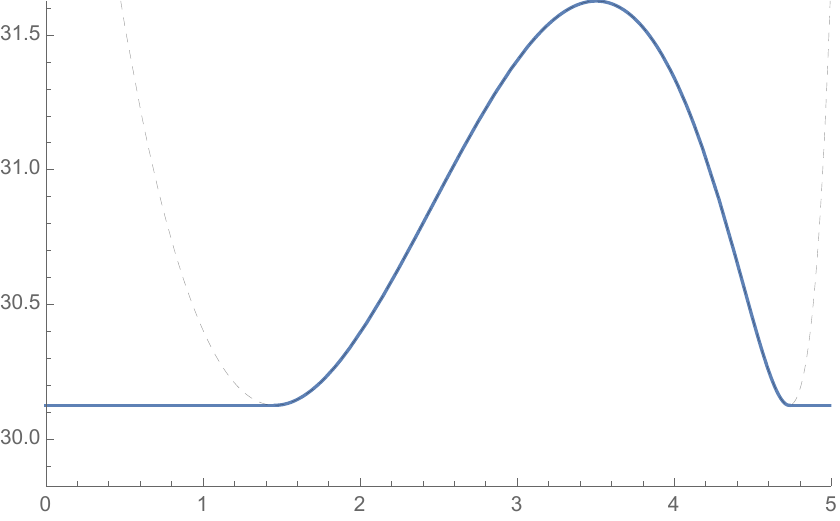}
		\caption{\footnotesize{$x \mapsto V(x)$ for Problem 3.}}
		\label{fig:0PICnonsym1}
	\end{minipage}
	\hfill 
	\begin{minipage}{0.45\textwidth}
		\centering
		\includegraphics[width=\textwidth]{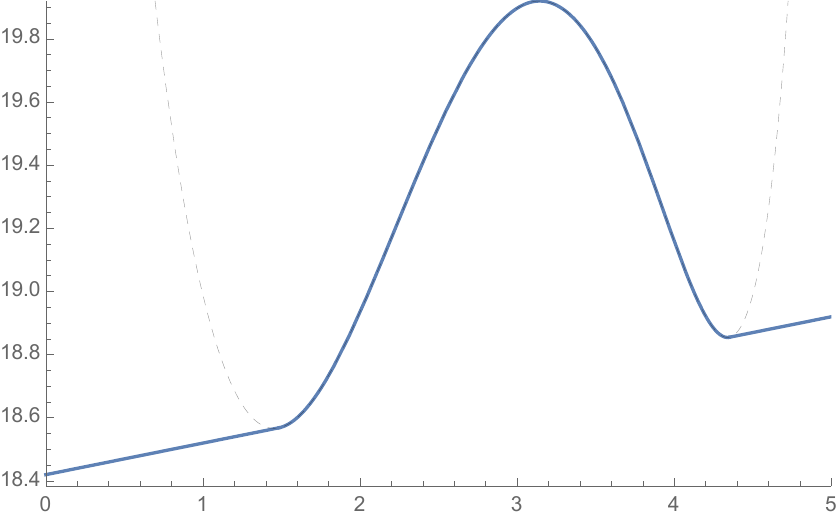}
		\caption{\footnotesize{$x \mapsto V(x)$ for Problem 4.}}
		\label{fig:0PICnonsym2}
	\end{minipage}
\end{figure}

\section{Conclusions}
\label{sec:conclusions}

In this paper, we look for an optimal price policy for a retailer selling energy to final consumers. The retailer buys the energy in the wholesale market and can adjust the final price he asks only by discrete-time interventions. His market share depends on the spread between the final price and the wholesale price, and each intervention corresponds to a fixed cost $c>0$. We model the problem as an infinite-horizon impulse control problem.

By a verification procedure, we characterize the value function and an optimal price policy: the retailer needs to intervene when the state variable exits from a fixed interval $]\x,\xx[$, moving the process to the most convenient state $x^*$. The value of $x^*$ is explicit, whereas $\x,\xx$ are characterized by a system of algebraic equations. We provide analytical results for the existence and uniqueness of solutions to this system.

We then focus on the role of the intervention cost $c>0$ in the optimal control. In particular, we provide an asymptotic estimate for the convergence of the continuation region $]\x,\xx[$ as $c \to 0^+$. Namely, the interval $]\x,\xx[$ converges to the singleton $\{x_v\}$, and its length converges to zero as the fourth root of the cost $c$. 

Finally, we propose some extensions to the problem, incorporating a drift term and variable intervention costs. Numerical results exemplify the properties we have proved.

To the best of our knowledge, the approach we present is original, and asymptotic estimates for impulse control problems have never been provided before. In particular, we would like to underline the scope and potentialities of impulse controls for optimization problems in energy markets. 

Several directions are possible to further develop the results here presented. For example, more structured models for the wholesale price can be considered (e.g., mean-reverting model for $S_t$), as well as different models for the market share function. In these cases, careful analyses are needed to study the properties of the optimal controls, as analytical existence results and semi-explicit formulas would no longer be possible. Furthermore, it would be interesting to consider the extension of the techniques in Alvarez \cite{Alvarez04-AMO} and Alvarez and Lempa \cite{AlvarezLempa08} to the case of two-sided impulse control problems, i.e., problems where the continuation region is a bounded interval.

\vspace{0.3cm}

\noindent\emph{Acknowledgements.} The author would like to thank René Aïd, Luciano Campi, Giorgia Callegaro, Tiziano Vargiolu and the anonymous referees for their valuable comments and suggestions.

\bibliographystyle{plain}

\begin{thebibliography}{1}

\bibitem{Aid15}{\sc R. Aïd}, {\em Electricity derivatives}, SpringerBriefs in Quantitative Finance, 2015.

\bibitem{AidBaseiCampiCallegaroVargiolu}{\sc R. Aïd, M. Basei, G. Callegaro, L. Campi, T. Vargiolu}, {\em Nonzero-sum stochastic differential games with impulse controls: a verification theorem with applications}, to appear in Math. Oper. Res. (2019)

\bibitem{Alvarez04-AMO}{\sc L. H. R. Alvarez}, {\em A class of solvable impulse control problems}, Appl. Math. Optim. 49 (2004), no. 3, 265--295.

\bibitem{Alvarez04}{\sc L. H. R. Alvarez}, {\em Stochastic forest stand value and optimal timber harvesting}, SIAM J. Control Optim. 42 (2004), no. 6, 1972--1992. 

\bibitem{AlvarezLempa08}{\sc L. H. R. Alvarez, J. Lempa}, {\em On the optimal stochastic impulse control of linear diffusions}, SIAM J. Control Optim. 47 (2008), no. 2, 703--732. 

\bibitem{BaseiCaoGuo}{\sc M. Basei, H. Cao, X. Guo}, {\em Nonzero-sum stochastic games with impulse controls}, ArXiv:1901.08085.

\bibitem{Belobaba87}{\sc P. P. Belobaba}, {\em Airline yield management: An overview of seat inventory control}, Transp. Sci. 21 (1987), no. 2, 63--73. 

\bibitem{BensoussanMoussawuCaka10}{\sc A. Bensoussan, L. Moussawi-Haidar, M. Cakanyildirim}, {\em Inventory control with an order-time constraint: optimality, uniqueness and significance}, Ann. Oper. Res. 181 (2010), no. 1, 603--640.

\bibitem{BenthKoek08}{\sc F. E. Benth, S. Koekebakker}, {\em Stochastic modeling of financial electricity contracts}, Energy Econ. 30 (2008), no. 3, 1116--1157.

\bibitem{CadenChoulliTaksarZhang06} {\sc A. Cadenillas, T. Choulli, M. Taksar, L. Zhang}, {\em Classical and impulse stochastic control for the optimization of the dividend and risk policies of an insurance firm}, Math. Finance 16 (2006), no. 1, 181--202.

\bibitem{CadenLaknerPinero10} {\sc A. Cadenillas, P. Lakner, M. Pinedo}, {\em Optimal control of a mean-reverting inventory}, Oper. Res. 58 (2010), no. 6, 1697--1710.

\bibitem{CadenSarkZap07} {\sc A. Cadenillas, S. Sarkar, F. Zapatero}, {\em Optimal dividend policy with mean-reverting cash reservoir}, Math. Finance 17 (2007), no. 1, 81--110.

\bibitem{CadenZapatero99} {\sc A. Cadenillas, F. Zapatero}, {\em Optimal Central Bank intervention in the foreign exchange market}, J. Econom. Theory 87 (1999), no. 1, 218--242.

\bibitem{Cosso13} {\sc A. Cosso}, {\em Stochastic differential games involving impulse controls and double-obstacle quasi-variational inequalities}, SIAM J. Control Optim. 51 (2013), no. 3, 2102--2131.

\bibitem{DavisGuoWu10} {\sc M. H. Davis, X. Guo, G. L. Wu}, {\em Impulse control of multidimensional jump diffusions}, SIAM J. Control Optim. 48 (2010), no. 8, 5276--5293.

\bibitem{ElmKesk03} {\sc W. Elmaghraby, P. Keskinocak}, {\em Dynamic pricing in the presence of inventory considerations: research overview, current practices, and future directions}, Manag. Sci. 49 (2003), no. 10, 1287--1309.

\bibitem{FedeRoseTacc2018} {\sc S. Federico, M. Rosestolato, E. Tacconi}, {\em Irreversible investment with fixed adjustment costs: a stochastic impulse control approach}, ArXiv:1801.04491.

\bibitem{FerrariKoch18} {\sc G. Ferrari, T. Koch}, {\em On a strategic model of pollution control}, Ann. Oper. Res. (2018).

\bibitem{GuoWu09} {\sc X. Guo, G. L. Wu}, {\em Smooth fit principle for impulse control of multi-dimensional diffusion processes}, SIAM J. Control Optim. 48 (2009), no. 2, 594--617.

\bibitem{HarrisonSelkeTaylor83}  {\sc M. J. Harrison, T. Selke, A. Taylor}, {\em Impulse control of a Brownian motion}, Math. Oper. Res. 8 (1983), no. 3, 454--466.

\bibitem{Jeanblanc93} {\sc M. Jeanblanc-Picqué}, {\em Impulse control method and exchange rate}, Math. Finance 3 (1993), no. 2, 161--177.

\bibitem{JeanblancShir95} {\sc M. Jeanblanc-Picqué, A. N. Shiryaev}, {\em Optimization of the flow of dividends}, Russian Math. Survey 50 (1995), no. 2, 257--277.

\bibitem{Korn98} {\sc R. Korn}, {\em Portfolio optimization with strictly positive transaction costs and impulse control}, Finance Stoch. 2 (1998), 85--114.

\bibitem{Korn99} {\sc R. Korn}, {\em Some applications of impulse control in mathematical finance}, Math. Meth. Oper. Res. 50 (1999), no. 2, 493--528.

\bibitem{McGillRyzin99}{\sc J. I. McGill, G. J. van Ryzin}, {\em Revenue management: Research overview and prospects}, Transp. Sci. 33 (1999), no. 2, 233--256. 

\bibitem{MitchellFengMuth14} {\sc D. Mitchell, H. Feng, K. Muthuraman}, {\em Impulse control of interest rates}, Oper. Res. 62 (2014), no. 3, 602--615.

\bibitem{MundacaOksendal98} {\sc G. Mundaca, B. Øksendal}, {\em Optimal stochastic intervention control with application to the exchange rate}, J. Math. Econom. 29 (1998), no. 2, 225--243.

\bibitem{Oksendal99} {\sc B. K. Øksendal}, {\em Stochastic control problems where small intervention costs have big effects}, Appl. Math. Optim. 40 (1999), no. 3, 355--375.

\bibitem{OksendalSulem07} 
{\sc B. K. {\O}ksendal, A. Sulem}, {\em Applied stochastic control of jump diffusions, Second Edition}, Springer-Verlag, Berlin-Heidelberg, 2007.

\bibitem{OksendalUboeZhang07} {\sc B. K. Øksendal, J. Ubøe, T. Zhang}, {\em Nonrobustness of some impulse control problems with respect to intervention costs}, Stoch. Anal. Appl. 20 (2002), no. 5, 999--1026.

\bibitem{Willassen98} {\sc Y. Willassen}, {\em The stochastic rotation problem: A generalization of Faustmann's formula to stochastic forest growth}, J. Econ. Dyn. Control 22 (1998), no. 4, 573--596.

\end{thebibliography}

\end{document}